\documentclass[12pt]{amsart}
\usepackage[usenames,dvipsnames]{xcolor}

\usepackage[margin=1in]{geometry}

\usepackage{amsmath}
\usepackage{amsfonts}
\usepackage{amssymb}
\usepackage{hyperref,graphicx}
\usepackage{aligned-overset}
\usepackage{bm}
\usepackage{pinlabel}
\usepackage{tikz-cd}

\newtheorem*{maintheorem*}{Main Theorem}

\newtheorem{theorem}{Theorem}[section]
\newtheorem{lemma}[theorem]{Lemma}
\newtheorem{corollary}[theorem]{Corollary}
\newtheorem{proposition}[theorem]{Proposition}

\theoremstyle{definition}
\newtheorem{definition}[theorem]{Definition}

\theoremstyle{remark}
\newtheorem{remark}[theorem]{Remark}

\numberwithin{equation}{section}

\newcommand{\F}{\mathbb {F}}
\newcommand{\R}{\mathbb {R}}

\newcommand{\Z}{\mathbb {Z}}

\newcommand{\dd}{\partial}

\newcommand{\alg}{\mathcal{A}}

\newcommand{\rep}{\mathcal{R}ep}

\newcommand{\wt}{\widetilde}

\newcommand{\aug}{\mathcal{A}ug}

\newcommand{\im}{\mathrm{im}\;}
\newcommand{\Hom}{\mathrm{Hom}}
\newcommand{\Mat}{\mathrm{Mat}}
\newcommand{\Span}{\mathrm{Span}}
\newcommand{\GL}{\mathrm{GL}}
\newcommand{\Aut}{\mathrm{Aut}}

\def\wt#1{\widetilde{#1}}

\title[The homotopy cardinality of the representation category]{The homotopy cardinality of the representation category for a Legendrian Knot}
\author{Justin Murray}
\address{Justin Murray, Department of Mathematics, Louisiana State University}
\email{jmurr24@lsu.edu}

\begin{document}
\maketitle
\begin{abstract}
    Given a Legendrian knot in $(\mathbb{R}^3, \ker(dz-ydx))$ one can assign a combinatorial invariants called ruling polynomials. These invariants have been shown to recover not only a (normalized) count of augmentations but are also closely related to a categorical count of augmentations in the form of the homotopy cardinality of the augmentation category. In this article, we prove that that the homotopy cardinality of the $n$-dimensional representation category is a multiple of the $n$-colored ruling polynomial. Along the way, we establish that two $n$-dimensional representations are equivalent in the representation category if they are ``conjugate homotopic''. We also provide some applications to Lagrangian concordance.
\end{abstract}
\section{Introduction}

Given a Legendrian knot $\Lambda$ in $\R^3$ with its standard contact structure given by $\ker(dz-ydx)$, one can assign to it a differential graded algebra (DGA), $(\alg_\Lambda,\dd_\Lambda)$, called the Chekanov-Eliashberg DGA. This DGA is generated by the Reeb chords of $\Lambda$ and whose differential counts  $J$-holomorphic disks with boundary on $\R\times \Lambda$ in the symplectization of $\R^3$. 
The stable-tame isomorphism class of $(\alg_\Lambda,\dd_\Lambda)$ is a powerful invariant of $\Lambda$ that can distinguish Legendrian knots, and be used to study Lagrangian cobordisms between Legendrians. Unfortunately, these tend to be hard to compute. What one can do to remedy this is to linearize the DGA by an augmentation $\epsilon:(\alg,\dd)\to(\F,0)$ to get the linearized contact homology $LCH^\epsilon_*$, which is easier to compute in general (but depends on $\epsilon$). For an in-depth discussion on Legendrian contact homology, see \cite{EN}.

Interestingly, augmentations have close ties to a purely combinatorial invariant called $m$-graded normal rulings discovered independently by Fuchs and by Chekanov and Pushkar in the context of generating families (see Definition \ref{ruling}) \cite{Fuchs, PC, FI, NS, NRSS, Sab}. In particular, $(\alg_\Lambda,\dd_\Lambda)$ has an $m$-graded augmentation if and only $\Lambda$ admits an $m$-graded normal ruling.
Moreover, in \cite{HR}, the authors were able to show that the $m$-graded ruling polynomial specialized at $q^{1/2}-q^{-1/2}$ returns the augmentation number $Aug_m(\Lambda,\F_q)$. This is a (normalized) count of $m$-graded augmentations into a finite field of order $q$. This implies that the ruling polynomial of $\Lambda$ is completely determined by $(\alg_\Lambda,\dd_\Lambda)$.


On the other hand, one can collect all of the augmentations into two types of $A_\infty$-categories, called positive and  negative augmentation categories. This idea of collecting augmentations into an $A_\infty$ category first appears in \cite{BC} where they introduce what is now called the negative augmentation category $\aug_-$. The positive augmentation category $\aug_+$ was constructed to show that there is an $A_\infty$-equivalence from $\aug_+(\Lambda,\F)$ to $\mathcal{S}h_1(\Lambda,F)$, the category of microlocal rank 1 constructible sheaves on $\R^2$ with microsupport on $\Lambda$. 
In both settings these are categories whose objects are augmentations and morphism spaces are certain forms of linearized contact cohomology \cite{NRSSZ}. In \cite{NRSS} the authors introduce a notion of cardinality for $\aug_+$ called the homotopy cardinality (a.k.a. multiplicative Euler characteristic \cite{BD01, MP07}) defined as 
\[\#\pi_{\geq0}\aug_+(\Lambda,\F_q)^*=\sum_{[\epsilon]\in\aug_+(\Lambda,\F_q)/\sim}\frac{1}{|\Aut(\epsilon)|}\cdot\frac{|H^{-1}\Hom(\epsilon,\epsilon)||H^{-3}\Hom(\epsilon,\epsilon)|\cdots}{|H^{-2}\Hom(\epsilon,\epsilon)||H^{-4}\Hom(\epsilon,\epsilon)|\cdots},\]
where here $|\Aut(\epsilon)|$ is the number of invertible elements in $H^0\Hom(\epsilon,\epsilon)$. 
Moreover, the main theorem of \cite{NRSS} asserts that the homotopy cardinality is actually computable via the ruling polynomial and the Thurston-Bennequin number: 
\[\#\pi_{\geq0}\aug_+(\Lambda,\F_q)^*=q^{tb(\Lambda)/2}R^0_{\Lambda}(z)\vert_{z=q^{1/2}-q^{-1/2}}\]
where $R^0_\Lambda$ is the 0-graded ruling polynomial. 
This homotopy cardinality has also been shown to be an effective obstruction to reversing certain Lagrangian cobordisms see \cite{CLL+}\cite{Pan}. 




\subsection{Main Results}Recently, \cite{LR,MR} defined $m$-graded $n$-colored ruling polynomials, $R^m_{n,\Lambda}(q) $ as a linear combination of ruling polynomials of various satellites of $\Lambda$ (see Definition \ref{coloredruling} for $m\neq1$) to obtain the similar counts of DGA maps into $(\Mat_n(\F_q),0)$. Such maps $\rho:(\alg_\Lambda,\dd_\Lambda)\to(\Mat_n(\F_q),0)$ are called \emph{n-dimensional representations}. Furthermore, in \cite{CDGG}, the augmentation category was generalized to the case where the target DGA is noncommutative. The goal of this article is to generalize the categorical aspect of the previously described story from augmentations to $n$-dimensional representations. 

\begin{maintheorem*}For a connected Legendrian $\Lambda\subset \R^3$ with $r(\Lambda)=0$ equipped with a single basepoint, we have 
    \[\#\pi_{\geq0}\rep_n(\Lambda,\F_q)^*=q^{n^2(\frac{tb-\chi_*}{2})}|\GL_n(\F_q)|^{-1}\cdot|\{\rho:\alg\to \Mat_n(\F_q)\}|
    \]
    where $\chi_*$ denotes the shifted Euler characteristic (defined in Equation \ref{eq: chi}) and the homotopy cardinality of the representation category is a multiple of the colored ruling polynomial: \[\#\pi_{\geq0}\rep_n(\Lambda,\F_q)^*=q^{n^2tb(\Lambda)/2}R^0_{n,\Lambda}(q).\]
\end{maintheorem*}

In Section \ref{conc} we provide some applications to Lagrangian concordance. In particular, a result of Pan \cite{Pan} implies that if there is an exact Lagrangian concordance from $\Lambda_- $ to $\Lambda_+$ then the homotopy cardinality of the augmentation category is increasing \[\#\pi_{\geq0}\aug_+(\Lambda_-,\F_q)^*\leq \#\pi_{\geq0}\aug_+(\Lambda_+,\F_q)^*.\] So if the inequality is not an equality, then this provides an obstruction to reversing the concordance. Section \ref{conc} establishes a similar result. Namely, we show if $r(\Lambda_\pm)=0$ and there is a concordance from $\Lambda_-$ to $\Lambda_+$, then \[\#\pi_{\geq0}\rep_n(\Lambda_-,\F_q)^*\leq \#\pi_{\geq0}\rep_n(\Lambda_+,\F_q)^*.\] We conjecture that this obstruction to reversing concordance is strictly stronger than that from the augmentation category and give a sketch of a proof, contingent on the existence of a knot $\Lambda_n$ that has a $0$-graded $n$-dimensional representation for some $n>1$, but no 0-graded augmentations.

For the remainder of the document, we will restrict our focus to Legendrian knots $\Lambda$ with $r(\Lambda)=0$ (which is required for the DGA of $\Lambda$ to have a $0$-graded $n$-dimensional representation). The document is organized as follows. In Section \ref{sec:repcat}, we give the background on the the $n$-dimensional representation category, $\rep_n(\Lambda,\F)$. Section \ref{sec:Equivrep}, is devoted to showing that two representations in $\rep_n(\Lambda,\F)$ are equivalent if and only if they are conjugate up to DGA homotopy. Section \ref{counting} establishes the main theorem listed above. After this, we conclude with Section \ref{conc}, which establishes some applications to concordance, and provides some further directions to explore.
\subsection*{Acknowledgements.} The author would like to thank his advisor Shea Vela-Vick for his support and many useful conversations over the years. We thank Dan Rutherford for numerous discussions and for suggesting the current problem. We also thank Mike Wong and Angela Wu for their helpful comments and advice. The author received partial support from NSF Grant DMS-1907654.

\section{The Representation Category}\label{sec:repcat}

In this section we review the construction of the (positive) representation category as defined in \cite{CDGG} and \cite{CNS18} with $A_\infty$ sign conventions coming from \cite{CNS18} and \cite{NRSSZ}. Although the objects and hom spaces are easily defined in the representation category, the interesting part of this $A_\infty$-category comes from its $A_\infty$ structure. We briefly summarize the definitions given in Section 2 of \cite{CNS18}, omitting some details.  

\subsection{From \texorpdfstring{$\alg_A$}{} to \texorpdfstring{$M^\vee$}{}} Let $\F$ be a field and $A:=\Mat_n(\F)$ and $\mathcal{R}:=\{a_1, \dots ,a_m\}$ a generating set equipped with gradings. Let $M_A$ be the free graded $A$-$A$-bimodule generated by $\mathcal{R}$ and having coefficient ring in grading 0. Set $M_A^{\boxtimes 0}:=A$ and define $\alg_A$ to be the graded $\F$-algebra given by \[\alg_A:=\mathcal{T}_A(M_A)=\bigoplus_{n=0}^\infty M_A^{\boxtimes n}\]
where $\boxtimes=\otimes_A$, and multiplication is defined by concatenation (and matrix multiplication). Suppose that in addition $\alg_A$ comes equipped with an differential $\dd_A$ satisfying the signed Leibniz rule and $\dd_A\vert_{A}=0$. As usual we will call $\dd_A$ \emph{augmented} if $\dd_A$ has no constant terms, that is, $\dd_A(\alg_A)\subset \bigoplus_{n=1}^\infty M_A^{\boxtimes n}$. 

Suppose $(\alg_A,\dd_A)$ is an augmented DGA. Let $M^\vee= \Hom_{A-A}(M_A,A)$ with grading determined by $\vert a_i^\vee\vert=\vert a_i\vert +1$. Observe that an element $\phi\in M^\vee$ is completely determined by where it sends elements of $\mathcal{R}$ and thereby can be identified with the free $A$-module generated by $\{a_1^\vee,\ldots,a_r^\vee\}$. A dualizing procedure of the $k$-th order part of the augmented differential $\dd_A$ gives $\F$-linear maps \[m_k:(M^\vee)^{\otimes_\F k}\to M^\vee\]
determined by the rule that a term in the augmented differential of the form 
\[\dd_A(b)=A_ka_{i_k}A_{k-1} a_{k-1}\cdots a_{i_1}A_0+\cdots\]
with $A_j\in A$, then contributes a term in $m_k$ as
\[ m_k(M_{i_1}a_{i_1}^\vee,\ldots,M_{i_k}a_{i_k}^\vee)=(-1)^\sigma (A_kM_{i_k}A_{k-1} M_{i_{k-1}}\cdots M_{i_1}A_0)b^\vee+\cdots\]
where 
\[\sigma=\frac{k(k-1)}{2}+\left(\sum_{p<q}|a_{i_p}^\vee||a_{i_q}^\vee|\right)+|a_{i_2}^\vee|+|a_{i_4}^\vee|+\cdots\]
More geometrically the terms of $m_k(M_{i_1}a_{i_1}^\vee,\ldots,M_{i_k}a_{i_k}^\vee)$ are obtained (up to sign) by replacing each of generators $a_{i_j}$ appearing in the augmented differential with the corresponding coefficient on $a_{i_j}^\vee$ and then recording the dual of the positive puncture (see Figure \ref{fig:phantom}).
We will later adopt shorthand notation $\langle m_k(M_{i_1}a_{i_1}^\vee,\ldots,M_{i_k}a_{i_k}^\vee),b^\vee \rangle$ for the coefficient of $b^\vee$ appearing in $m_k(M_{i_1}a_{i_1}^\vee,\ldots,M_{i_k}a_{i_k}^\vee)$.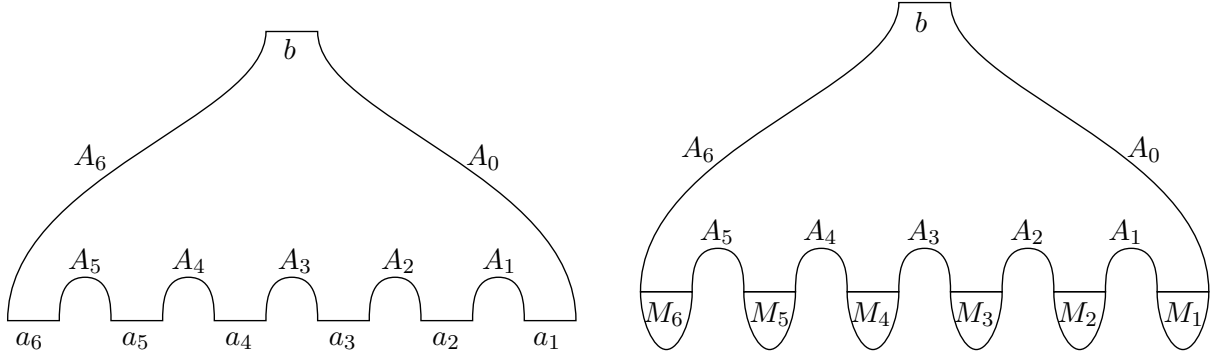
\begin{figure}[htb]
\labellist
\small\hair 2pt
 \pinlabel {$A_6$} [ ] at 33 74
 \pinlabel {$a_6$} [ ] at 10 5
 \pinlabel {$a_5$} [ ] at 50 5
 \pinlabel {$A_5$} [ ] at 31 35
 \pinlabel {$A_4$} [ ] at 70 35
 \pinlabel {$a_4$} [ ] at 89 5
 \pinlabel {$A_3$} [ ] at 110 35
 \pinlabel {$a_3$} [ ] at 128 5
 \pinlabel {$A_2$} [ ] at 149 35
 \pinlabel {$a_2$} [ ] at 167 5
 \pinlabel {$A_1$} [ ] at 187 35
 \pinlabel {$a_1$} [ ] at 205 5
 \pinlabel {$A_0$} [ ] at 181 74
 
 \pinlabel {$A_6$} [ ] at 262 77
 \pinlabel {$M_6$} [ ] at 249 15
 \pinlabel {$A_5$} [ ] at 269 46
 \pinlabel {$M_5$} [ ] at 289 15
 \pinlabel {$A_4$} [ ] at 308 46
 \pinlabel {$M_4$} [ ] at 327 15
 \pinlabel {$A_3$} [ ] at 347 46
 \pinlabel {$M_3$} [ ] at 366 15
 \pinlabel {$A_2$} [ ] at 386 46
 \pinlabel {$M_2$} [ ] at 405 15
 \pinlabel {$A_1$} [ ] at 423 46
 \pinlabel {$M_1$} [ ] at 445 15
 \pinlabel {$A_0$} [ ] at 430 77
 \pinlabel {$b$} [ ] at 108 115
 \pinlabel {$b$} [ ] at 346 125
\endlabellist
\centering
\includegraphics[scale=1.0]{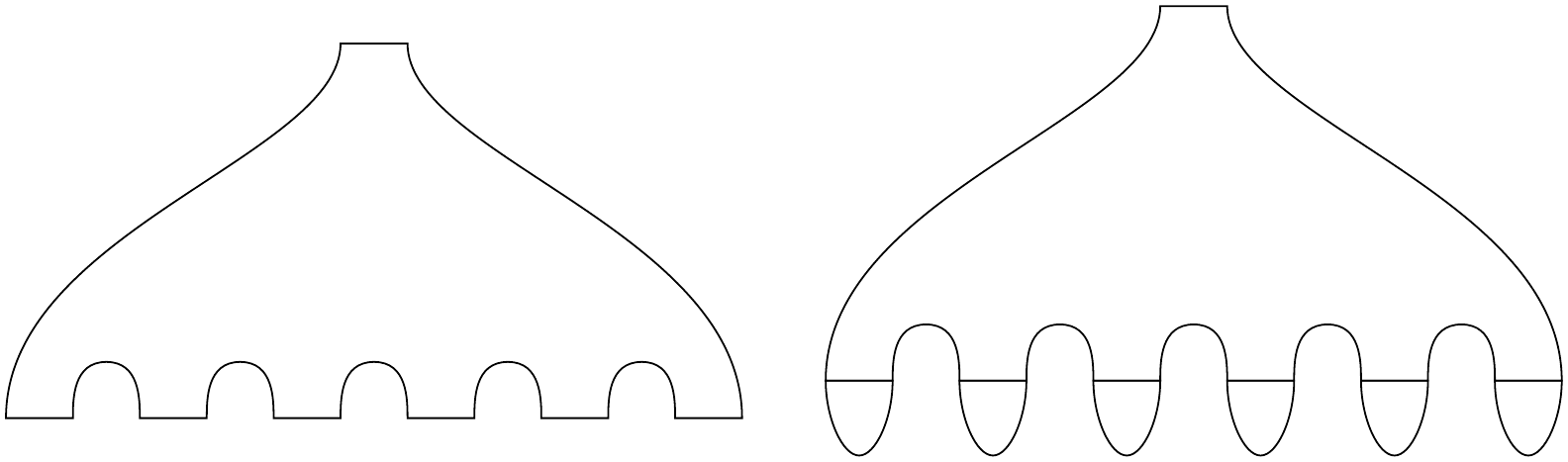}
\caption{Left: A geometric picture of an augmented disk contributing to $\pm A_6a_6A_5a_5A_4a_4A_3a_3A_2a_2A_1a_1A_0$ in $\dd_A(b)$ (with the $A_6$ labelled strand corresponding to the overstrand of $b$). Right: The corresponding term contributing to the coefficient $\langle m_6(M_1a_1^\vee,M_2a_2^\vee,M_3a_3^\vee,M_4a_4^\vee,M_5a_5^\vee,M_6a_6^\vee),b^\vee\rangle$}
\label{fig:phantom}
\end{figure}
The pair $(M^\vee,\{m_k\})$ then forms an $A_\infty$ algebra as described in \cite{CDGG} and \cite{CNS18}. 
\subsection{Constructing augmented differentials on \texorpdfstring{$\alg_A$}{} from representations} Let $R$ be another $\F$-algebra and $\alg_{R}$ be another DGA (constructed as a tensor algebra as before) 
equipped with a differential $\dd_R$ satisfying the Leibnez rule and $\dd_R\vert_R=0$. For us, this is going to be the fully-noncommutative Chekanov-Eliashberg DGA of a Legendrian knot equipped with a single basepoint tensored with $\F$ (so in particular $R=\F[t^{\pm 1}$]). 

\begin{definition}
    An \emph{$n$-dimensional representation} of $(\alg_R,\dd_R)$ to $A$ is a unital DGA map
    \[\rho:(\alg_R,\dd_R)\to (A,0)\]
    (in particular $\rho\circ\dd=0$ and $\rho(a_i)=0$ unless $|a_i|=0$).
\end{definition}
\begin{remark}
     The previous definition is actually the definition a $0$-graded representation in the sense of \cite{LR} as it has support on 0-graded elements of $\mathcal{R}$. Furthermore, if the DGA of $\Lambda$ has 0-graded $n$-dimensional representation, then $\Lambda$ must have $r(\Lambda)=0$. \emph{Therefore, we will assume, for the rest of the paper, that $r(\Lambda)=0$, as there is no real loss in generality}.
\end{remark}

Given a representation $\rho:(\alg_R,\dd_R)\to (A,0)$, one gets an induced differential on $\alg_A$ by simply replacing elements of $R$ appearing in $\dd_R$ with their image under $\rho$. Letting $\dd_A$ be the induced differential one can twist $\dd_A$ by $\rho$ to get an augmented differential $\dd_A^\rho$ on $\alg_A$. 
In particular, define the ``twisting'' $\F$-algebra homomorphism on generators by $\phi_\rho:\alg_R\to\alg_A$ by
\[\phi_\rho(a)=\begin{cases}
    a+\rho(a) &a\in\mathcal{R}\\
    \rho(a) &a\in R
\end{cases}\]
and then extend $\phi_\rho$ on all of $\alg_R$ by the signed Leibniz rule. Then one gets an augmented differential $\dd_A^\rho:\dd_A\to\dd_A$ by defining 
\[\dd_A^\rho(a)=\begin{cases}
    \phi_\rho(\dd_R(a)) & a\in\mathcal{R}\\
    0 & a\in A
\end{cases}\] 
on generators and extending by the signed Leibniz rule. With an augmented differential in hand one can construct $(M^\vee,\{m_k\}_{k\geq1})$ as in the previous subsection.
\subsection{The representation category}
As in the case of augmentation categories, there are two flavors for the $n$-dimensional representation categories a ``negative'' version and a ``positive'' version c.f. \cite{BC} \cite{CNS18} and \cite{NRSSZ}. The negative version is the $A_\infty$ category whose objects are $n$-dimensional representations and whose self hom-spaces form the $A_\infty$ algebra defined earlier in this section. A downside of the negative version is that it fails to to have a strict unit. In contrast, the positive representation category does have a strict unit given by $-y^\vee$ (see \cite{CNS18}). The existence of this unit is essential to the remainder of this document and so we will drop the positive adjective with understanding that we are talking about the positive representation category.

\begin{definition}
    Let $\Lambda$ be a Legendrian knot equipped with a basepoint and $(\alg_\Lambda,\dd_\Lambda)$ the Chekanov-Eliashberg DGA. The representation category $\rep_n(\Lambda,\F)$ is the $A_\infty$ category defined by
    \begin{itemize}
        \item[1.] Objects: $n$-dimensional representations $\rho:(\alg,\dd)\to(\Mat_n(\F),\dd)$
        \item[2.] Morphisms: Given two representations $\rho_1,\rho_2$ we define $\Hom(\rho_1,\rho_2)$ to be the free graded $\Mat_n(\F)$-module generated by $a_1^\vee,\ldots,a_m^\vee, x^\vee, y^\vee$ where $a_i^\vee$ are formal duals to Reeb chords along with two extra generators. Gradings are given by $|a_i^\vee|=|a_i|+1$, $|x^\vee|=1$ and $|y^\vee|=0$
        \item[3.] Compositions: The $\F$-multilinear maps \[m_k:\Hom(\rho_k,\rho_{k+1})\otimes\Hom(\rho_{k-1},\rho_k)\otimes\cdots\otimes\Hom(\rho_1,\rho_2)\to\Hom(\rho_1,\rho_{k+1})\] are defined by dualizing the augmented differential of the $(k+1)$-copy of $\Lambda$ with respect to the \emph{pure} representation $\bm{\rho}=(\rho_1,\ldots,\rho_{k+1})$ as before.
    \end{itemize}
\end{definition}
We now flesh out 3 in more detail. We describe the pure representation after we review the DGA of the Lagrangian projection $k+1$-copy $\Lambda^{k+1}_f$ described in \cite{NRSSZ}. Let $f$ be a perturbing Morse function with exactly 1 maximum followed by exactly 1 minimum occurring within a small neighborhood of the basepoint (here the  ordering is prescribed via the orientation of $\Lambda$). Using $f$ one can obtain a Lagrangian projection of the $k+1$-copy, $\Lambda^{k+1} _f$, to create $k+1$ parallel copies of $\Lambda$ (see \cite[Section 4]{NRSSZ} for more details). Diagrammatically this looks like a blackboard framed $k+1$ copy of $\pi_{xy}(\Lambda)$ with a single ``dip'' near the orginial basepoint on $\Lambda$ as depicted in Figure \ref{fig:3copy}.

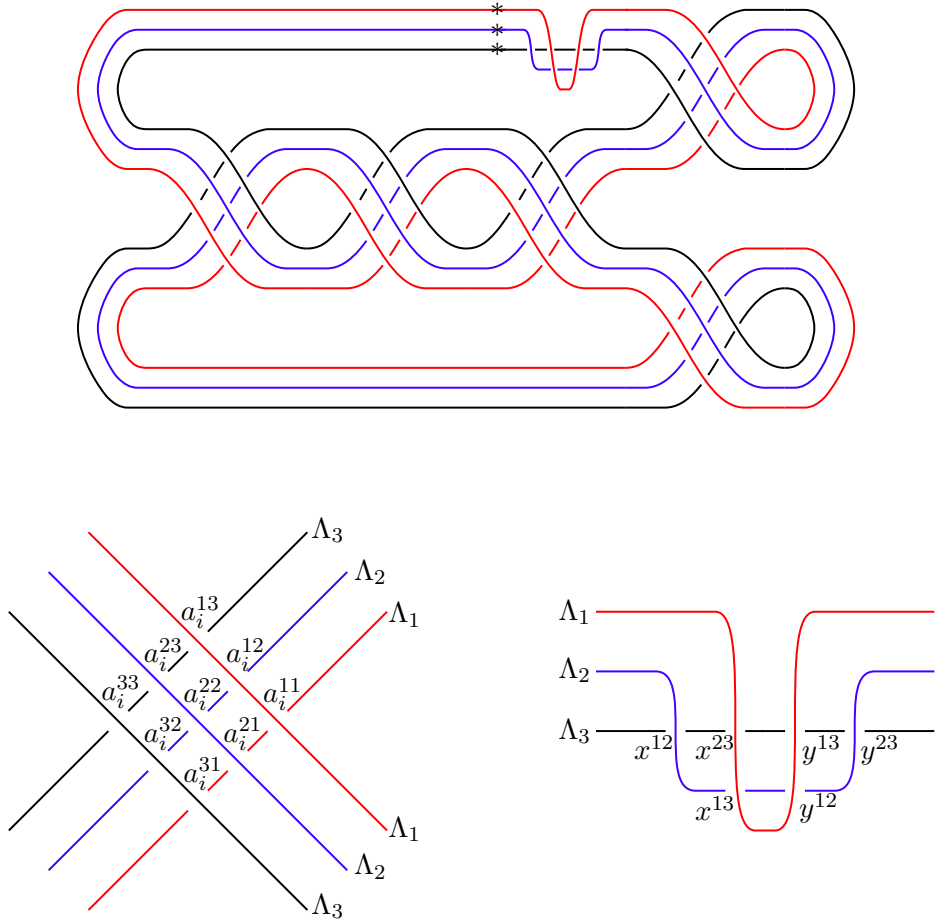
\begin{figure}[ht]
\labellist
\small\hair 2pt
 \pinlabel {$*$} [ ] at 160 152
 \pinlabel {$*$} [ ] at 160 144
 \pinlabel {$*$} [ ] at 160 137
\endlabellist
\centering
\includegraphics[scale=1.0]{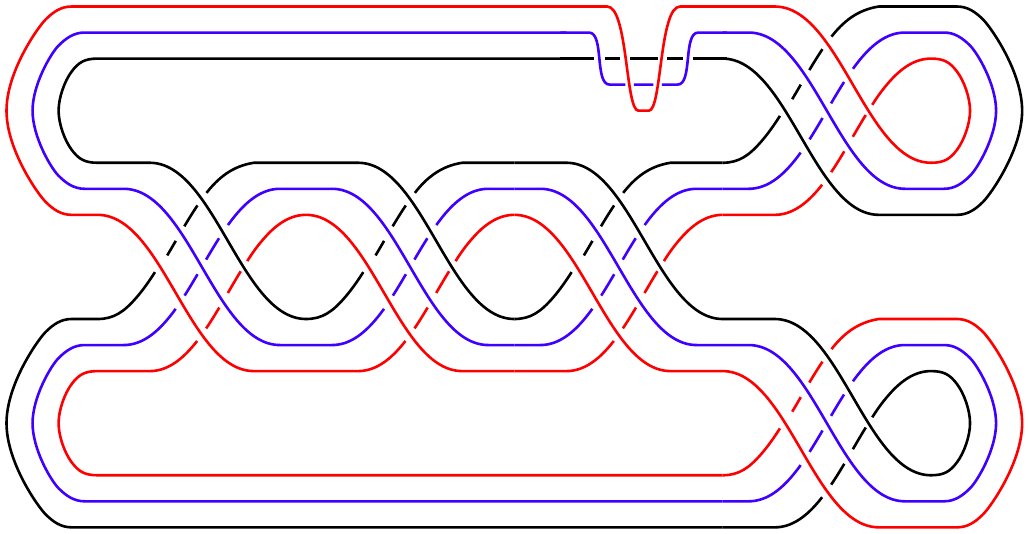}
\vspace{1cm}
\labellist
\small\hair 2pt
 \pinlabel {$a_i^{13}$} [ ] at 74 114
 \pinlabel {$a_i^{12}$} [ ] at 91 98
 \pinlabel {$a_i^{11}$} [ ] at 105 83
 \pinlabel {$a_i^{21}$} [ ] at 90 68
 \pinlabel {$a_i^{22}$} [ ] at 75 83
 \pinlabel {$a_i^{23}$} [ ] at 60 98
 \pinlabel {$a_i^{33}$} [ ] at 44 84
 \pinlabel {$a_i^{32}$} [ ] at 60 68
 \pinlabel {$a_i^{31}$} [ ] at 75 53
 \pinlabel {$x^{12}$} [ ] at 245 63
 \pinlabel {$x^{23}$} [ ] at 268 63
 \pinlabel {$x^{13}$} [ ] at 268 40
 \pinlabel {$y^{12}$} [ ] at 307 40
 \pinlabel {$y^{13}$} [ ] at 308 62
 \pinlabel {$y^{23}$} [ ] at 330 62
 \pinlabel {$\Lambda_1$} [ ] at 151 113
 \pinlabel {$\Lambda_2$} [ ] at 138 129
 \pinlabel {$\Lambda_3$} [ ] at 122 145
 \pinlabel {$\Lambda_1$} [ ] at 151 32
 \pinlabel {$\Lambda_2$} [ ] at 138 18
 \pinlabel {$\Lambda_3$} [ ] at 122 3
 \pinlabel {$\Lambda_1$} [ ] at 215 114
 \pinlabel {$\Lambda_2$} [ ] at 215 92
 \pinlabel {$\Lambda_3$} [ ] at 215 69
\endlabellist
\vspace{.5cm}
\centering
\includegraphics[scale=1.0]{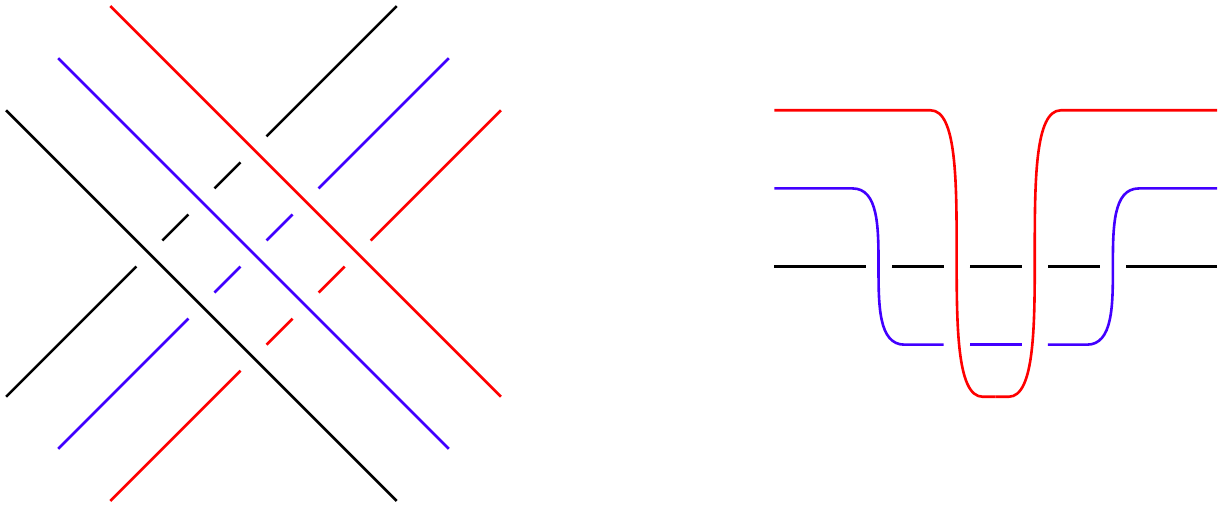}
\caption{Top: The Lagrangian 3-copy of the trefoil. Bottom: Crossing and dip generators}
\label{fig:3copy}
\end{figure}
Label the components of the copy as $\Lambda_1$ to $\Lambda_{k+1}$ starting from the top and working down \emph{near the basepoint}. We label Reeb chords of the perturbed $k+1$-copy by $a_k^{i,j}$, $x^{i,j}$, and $y^{i,j}$ where the superscript $i,j$ denotes a Reeb chord from $\Lambda_j$ to $\Lambda_i$ (with $\Lambda_j$ under $\Lambda_i$). Here $x^{i,j}$ (resp. $y^{i,j}$) are the Reeb chords produced from the maximum (resp. minimum) of $f$ during the perturbation process. We will call a Reeb chord of the $k+1$-copy \emph{pure} if $i= j$ and \emph{mixed} otherwise.
Adopting this labeling scheme, one can describe the DGA of the $k+1$-copy explicitly using matrices.
\begin{proposition}\cite[Proposition 4.14]{NRSSZ}
    Let $(\alg,\dd)$ be the DGA of $\Lambda$ generated by Reeb chords $a_1,\ldots,a_m$ and invertible generators $t^{\pm1}$ corresponding to the basepoint. Then the DGA of the $k$-copy, denoted $(\alg_{\Lambda^k_f},\dd_{\Lambda^k_f})$ is generated by:
    \begin{itemize}
        \item $a^{i,j}_\ell$ for $1\leq i,j\leq k$ and $1\leq\ell\leq m$, with $|a^{i,j}_\ell|=|a_\ell|$ 
        \item $x^{i,j}$ for $1\leq i<j\leq k$, with $|x^{i,j}|=0$
        \item $y^{i,j}$ for $1\leq i<j\leq k$, with $|y^{i,j}|=-1$
        \item $t^{\pm1}_i$ for $1\leq i \leq k$, with $|t^{\pm1}_i|=0$
    \end{itemize}
    The differential of the $k$-copy is described as follows. Define $k\times k$ matrices $A_1,\ldots A_m$, $X,Y$, and $\Delta$
    by $(A_\ell)_{ij}=a_\ell^{ij}$ and
    \[ (X)_{ij}=\begin{cases}
        1 & i=j\\
        x^{i,j} &i<j\\
        0 & j>i
    \end{cases},\qquad (Y)_{ij}=\begin{cases}
        y^{i,j} &i<j\\
        0 &i\geq j
    \end{cases},\qquad
    (\Delta)_{ij}=\begin{cases}
        t_i &i=j\\
        0 &i\neq j.
    \end{cases}\] Then, the differential of the $k$-copy is given entry-by-entry by
    \begin{align*}
        \dd_{\Lambda^k_f}(A_\ell)&=\Phi(\dd(a_\ell))+YA_\ell-(-1)^{|a_\ell|}A_\ell Y\\
        \dd_{\Lambda^k_f}(X)&=\Delta^{-1}Y\Delta X-XY\\
        \dd_{\Lambda^k_f}(Y)&=Y^2
    \end{align*}where $\Phi:\alg\to \Mat_k(\alg_{\Lambda^k_f})$ defined by $\Phi(a_\ell)=A_\ell$, $\Phi(t)=\Delta X$, and $\Phi(t^{-1})=X^{-1}\Delta^{-1}$.
\end{proposition}  

Lastly augmenting the $k+1$-copy DGA with respect to pure representation $\bm{\rho}$ means that we twist the differential with respect to the representation given by $\bm{\rho}(a_\ell^{i,j})=\delta_{ij}\rho_i(a_\ell)$ and $\bm{\rho}(t_i)=\rho_i(t)$. The construction of $m_k$ then follows from above after restricting to \emph{increasing} composable words (c.f. \cite[Lemma 3.15]{NRSSZ}), that is if $c^{i,j}_\ell$ is any of the $a_k^{i,j}$, or $x^{i,j}$ or $y^{i,j}$ generators then
\[\dd^{\bm{\rho}}_A(b^{1,k})=A_kc^{1,2}_{k}A_{k-1} c^{2,3}_{k-1}\cdots c^{k-1,k}_{1}A_0+\cdots\] gives rise to a contribution of 
\[m_k(M_{1}c_{1}^\vee,\ldots,M_{k}c_{k}^\vee)=(-1)^\sigma (A_kM_{k}A_{k-1} M_{k-1}\cdots M_{1}A_0)b^\vee+\cdots.\]
\section{Equivalence in the Representation Category}\label{sec:Equivrep}
We will call two representations in $\rep_n$ 
isomorphic if they are isomorphic in their associated cohomology category $H^*\rep_n$. In the augmentation category, one can show (as done in \cite{NRSSZ}) that this notion of equivalence is equivalent to two augmentations being DGA homotopic. However the augmentation category has the property that each $m_k$ is $\F$-linear whereas in the representation category there is no $\Mat_n(\F)$-linearity. On the other hand, we know from representation theory that two linear representations of a group are isomorphic if there is a conjugation relating them. Naturally, one might hope that combining these two perspectives yields an equivalent notion of isomorphic representations in $\rep_n$. Indeed, in this section we show that two representations are isomorphic if and only if they are ``conjugate DGA homotopic'' (Definition \ref{conghom} below). 

The following two ``characterization lemmas'' follow directly from the description of the DGA of the $n$-copy and the definition of $\rep_n$ and will be used throughout this section. The reader is encouraged to compare this with \cite[Lemmas 5.16 and 5.17]{NRSSZ}
\begin{lemma}[Characterization of $m_1$]\label{charm1}
    In $\Hom(\rho_1,\rho_2)$, we have
    \begin{align*} m_1(Ma_i^\vee)&=\sum_j\sum_{\substack{a_j,b_1,\ldots b_n \\ u\in\Delta (a_j;b_1,\ldots,b_n)}} \sum_{1\leq\ell \leq n}\delta_{b_\ell,a_i}\sigma_u \rho_1(b_1\cdots b_{\ell-1})M\rho_2(b_{\ell+1}\cdots b_n)a_j^\vee\\
        m_1(My^\vee)&=(\rho_1(t)^{-1}M\rho_2(t)-M)x^\vee+\sum_j(M\rho_2(a_j)-\rho_1(a_j)M)a_j^\vee\\
        m_1(Mx^\vee)&\in \Span_{\Mat_n(\F)}\{a_j^\vee\}
    \end{align*}
    where $\sigma_u$ denotes the product of orientation signs of $u$, and $\Delta(a_j; b_1\ldots b_n)$ denotes the moduli space of holomorphic disks with a single positive puncture at $a_j$ and the $b_i$'s are negative punctures or any integer power of the basepoint generator $t$ coming from intersections of the basepoint and the boundary of $u$.
\end{lemma}
\begin{lemma}[Characterization of $m_2$]\label{charm2}
    Assume that the Reeb chords of $\Lambda$ are labeled by increasing height, $h(a_1)<h(a_2)<\cdots<h(a_m)$. Then for any representations $\rho_1,\rho_2$, and $\rho_3$, the map $m_2:\Hom(\rho_2,\rho_3)\otimes\Hom(\rho_1,\rho_2)\to\Hom(\rho_1,\rho_3)$ satisfies  
    \begin{itemize}
        \item $m_2(Ma_i^\vee,Na_j^\vee)\in \Span_{\Mat_n(\F)}\{a_\ell^\vee | \ell>\max\{i,j\}\}$
        \item $m_2(Mx^\vee,Na_i^\vee)$, $m_2(Ma_i^\vee,Nx^\vee)$, and $m_2(Mx^\vee,Nx^\vee)$ are all in $\Span_{\Mat_n(\F)}\{a_j^\vee\}$
        \item $m_2(y^\vee,\alpha)=m_2(\alpha,y^\vee)=-\alpha$ for any $\alpha\in \Span_{\Mat_n(\F)}\{y^\vee,x^\vee,a_1^\vee,\ldots, a_m^\vee\}$
        \item
        For $\alpha=M_\alpha y^\vee+A$ and $\beta=M_\beta y^\vee+B$ with $A,B\in\Span_{\Mat_n (\F)}\{a_i^\vee\vert \;|a_i|=-1\}$ the following holds: \begin{align*}
        m_2(\alpha,\beta)&=m_2(M_\alpha y^\vee,M_\beta y^\vee)+m_2(M_\alpha y^\vee,B)+m_2(A,M_\beta y^\vee)+m_2(A,B)\\
            &=-M_\beta M_\alpha y^\vee -B \cdot M_\alpha-M_\beta A+m_2(A,B).\end{align*}
        Where  $B\cdot M_\alpha$ denotes right multiplication by $M_\alpha$ on all coefficients so if $B=\sum_iB_ia_i^\vee$ then $B\cdot M_\alpha=\sum_iB_iM_\alpha a_i^\vee$.    
    \end{itemize}
\end{lemma}

\begin{remark}
    Comparing Lemma \ref{charm2} with Lemma 5.17 in \cite{NRSSZ} the final bullet point seems somewhat mysterious to include, but these four computations are the only four that we will need.
\end{remark}

\begin{lemma}[Cocycle condition]\label{cocycle}
    Any element of the form $\alpha=My^\vee-\sum_iK(a_i)a_i^\vee\in \Hom^0(\rho_1,\rho_2)$ with $M\in \GL_n(\F)$, and $K(a_i)\in \Mat_n(\F)$ is a cocycle if and only if $M\rho_2(t)=\rho_1(t) M$ and $M\rho_2(a_i)-\rho_1(a_i) M=\wt{K}\circ\dd(a_i)$ holds for all $i$. Here $\wt{K}$ is the unique $(\rho_1,\rho_2)$-derivation from $\alg$ to $\Mat_n(\F)$ extending $K$.
\end{lemma}

\begin{proof} We just compute $m_1(\alpha)$ (with a mild abuse of notation: allowing the $b_i$'s to denote the invertible $t$ generators)
    \begin{align*}
    m_1(\alpha)&=m_1(My^\vee)-m_1\left(\sum_iK(a_i)a_i^\vee\right)\\
    &=(\rho_1(t)^{-1}M\rho_2(t)-M)x^\vee+\sum_j(M\rho_2(a_j)-\rho_1(a_j)M)a_j^\vee\\
    &\quad-\sum_j\sum_{\substack{a_j,b_1,\ldots b_n \\ u\in\Delta (a_j;b_1,\ldots,b_n)}} \sum_{\substack{1\leq\ell \leq n\\ b_\ell\neq t^{\pm1}}}\sigma_u \rho_1(b_1\cdots b_{\ell-1})K(b_\ell)\rho_2(b_{\ell+1}\cdots b_n)a_j^\vee\\
    &=(\rho_1(t)^{-1}M\rho_2(t)-M)x^\vee+\sum_j(M\rho_2(a_j)-\rho_1(a_j)M)a_j^\vee\\
    &\quad-\sum_j\sum_{\substack{a_j,b_1,\ldots b_n \\ u\in\Delta (a_j;b_1,\ldots,b_n)}} \sum_{\substack{1\leq\ell \leq n\\ b_\ell\neq t^{\pm1}}}(-1)^{\vert b_1\cdots b_{\ell-1}\vert}\sigma_u \rho_1(b_1 \cdots b_{\ell-1})K(b_\ell)\rho_2(b_{\ell+1}\cdots b_n)a_j^\vee\\
    &=(\rho_1(t)^{-1}M\rho_2(t)-M)x^\vee+\sum_j(M\rho_2(a_j)-\rho_1(a_j)M-\wt{K}\circ\dd(a_j))a_j^\vee
     \end{align*}
Where at the last equality we used induction after including the vanishing terms coming from $b_\ell=t^{\pm1}$ (we have that $\wt{K}(t^{\pm1})=0$ for grading reasons). The claim now follows.
\end{proof}
\begin{remark}
    The above computation in Lemma \ref{cocycle} holds if the assumption that $M$ is invertible is dropped. However, we will see that invertibility of $M$ is needed for two representations to be equivalent in $\rep_n(\Lambda,\F)$.
\end{remark}
\begin{remark}\label{expo}    
If $M\rho_1$ and $\rho_2 M$ were DGA maps (which they are not), then having $M\rho_2(a_i)-\rho_1(a_i) M=\wt{K}\circ\dd(a_i)$ with $\wt{K}$ being a $(\rho_1,\rho_2)$-derivation is the definition of a DGA homotopy between $M\rho_2$ and $\rho_1 M$ and we would be tempted to write $M\rho_2\sim \rho_1M$. This can be fixed in the following way. First observe that $\rho_1$ and $M\rho_2M^{-1}$ are DGA maps. When $m_1(\alpha)=0$, then multiplying the coefficients of $a_j^\vee$ by $M^{-1}$ on the right gives that
\[\sum_j (M\rho_2(a_j)M^{-1}-\rho_1(a_j))a_j^\vee\overset{\star}{=}\sum_{\textrm{disks}}[(-1)^{\vert b_1\cdots b_{\ell-1}\vert}\sigma_u \rho_1(b_1 \cdots b_{\ell-1})K(b_\ell)\rho_2(b_{\ell+1}\cdots b_n)]M^{-1}a_j^\vee\]
where the $\sum_{\textrm{disks}}$ should be interpreted as the triple summation in the proof of Lemma \ref{cocycle}. Define $\wt{L}$ to be the unique $(\rho_1,M\rho_2M^{-1})$-derivation extending $L(x)=K(x)M^{-1}$. Then by definition $L(xy)=L(x)M\rho_2(y)M^{-1}+(-1)^{|x|}\rho_1(x)L(y)=[K(x)\rho_2(y)+(-1)^{|x|}\rho_1(x)K(y)]M^{-1}$. Then by induction, this gives that the coefficient of $a_j^\vee$ on the right hand side of equality $\star$ is exactly $\sum_j\wt{L}\circ \dd(a_j)$.
\end{remark}
\begin{definition}\label{conghom}
    Given two representations $\rho_1,\rho_2:\alg\to \Mat_n(\F)$, we say that $\rho_1$ and $\rho_2$ are \emph{conjugate DGA homotopic}, if there is $M\in \GL_n(\F)$ and a $(\rho_1, M\rho_2M^{-1})$-derivation $L$ so that $M\rho_2(t)M^{-1}=\rho_1(t)$ and $M\rho_2M^{-1}-\rho_1={L}\circ\dd$.  We will write $M\rho_2M^{-1}\sim\rho_1$ to mean a conjugate DGA homotopy 
\end{definition}

\begin{lemma}
    Conjugate DGA homotopy is an equivalence relation
\end{lemma}
\begin{proof}
    This follows from the fact that $M$ is invertible and that usual DGA homotopy is an equivalence relation which can be found in \cite{FHT}. 
\end{proof} 

We are now at the stage to prove when two representations are equivalent. As in the augmentation category, we define two representations $\rho_1$, $\rho_2$ in $\rep_n(\Lambda,\F)$ to be equivalent, if they are isomorphic in the cohomology category $H^*\rep_n(\Lambda,\F)$, (c.f. \cite{NRSSZ} \cite{CNS18}). 
\begin{proposition}\label{equiv}
    Two representations in $\rep_n(\Lambda,\F)$ are equivalent in $H^*\rep_n(\Lambda,\F)$ if and only if they are conjugate DGA homotopic
\end{proposition}

\begin{proof}
    Suppose $\rho_1\cong\rho_2$ in $\rep_n(\Lambda,\F)$. Then there exists cocycles $\alpha\in \Hom^0(\rho_1,\rho_2)$ and $\beta\in \Hom^0(\rho_2,\rho_1)$ such that $[m_2(\alpha,\beta)]=[-y^\vee]$ in $H^0\Hom(\rho_2,\rho_2)$ 
    Then, we have $m_2(\alpha,\beta)=-y^\vee+m_1(\gamma)$ for some $\gamma$ in grading $-1$. If we denote by $\langle m_1(\gamma),y^\vee\rangle$ the coefficient of $y^\vee$ appearing in $m_1(\gamma)$, we deduce that
        $\langle m_1(\gamma),y^\vee\rangle=0$ (by characterization of $m_1$ in Lemma \ref{charm1}) and that
        $\langle m_1(\gamma),x^\vee\rangle=0$ (for grading reasons). So we can rewrite 
        \begin{equation}
        m_2(\alpha,\beta)=-y^\vee+\sum_{|a_i|=-1}K(a_i)a_i^\vee.\label{m21}
        \end{equation}
        Since $\alpha, \beta $ are in grading 0 and are cocycles, we only need to verify that the coefficient on $y^\vee$ appearing in $\alpha$ is invertible, since Lemma \ref{cocycle} and Remark \ref{expo} would complete the proof. Towards this end, let $\alpha=M_\alpha y^\vee +A$ and $\beta=M_\beta y^\vee+B$ for $A,B\in \Span_{\Mat_n (\F)}\{a_i^\vee\vert \;|a_i|=-1\}$. Then, since $m_k$ is $\F$-multilinear we get
        \begin{align*}
            m_2(\alpha,\beta)&=m_2(M_\alpha y^\vee,M_\beta y^\vee)+m_2(M_\alpha y^\vee,B)+m_2(A,M_\beta y^\vee)+m_2(A,B)\\
            &=-M_\beta M_\alpha y^\vee -B \cdot M_\alpha-M_\beta A+m_2(A,B).\tag{3.2}\label{m22}
        \end{align*}
Comparing \ref{m21} and \ref{m22} and using Lemma \ref{charm2}, we see that $M_\alpha\in \GL_n\F$ as required.

Conversely, suppose that there is an $M\in \GL_n(\F)$ and a $(\rho_1,M\rho_2M^{-1})$-derivation $L$ such that $M\rho_2M^{-1}-\rho_1 =L\circ \dd$. Since we are viewing $(\Mat_n(\F),0)$ as a DGA concentrated in degree 0, we see the only nonzero $L(a_i)$ must come from those $a_i$ in DGA degree $-1$. Therefore, by Lemma \ref{cocycle} and Remark \ref{expo}, we see $\alpha=My^\vee +\sum_iL(a_i)a_i^\vee$ is a cocycle in $\Hom^0(\rho_1,\rho_2)$. We will find a $\beta,\gamma\in \Hom^0(\rho_2,\rho_1)$ such that $m_2(\alpha,\beta)=m_2(\gamma,\alpha)=-y^\vee$ and then once we show $\beta$ and $\gamma$ are cocycles it will follow that $[m_2(\alpha,\beta)]=[m_2(\gamma,\alpha)]=-[y^\vee]$. This will give us $\rho_1\cong \rho_2$.

The constructions of $\beta$ and $\gamma$ are symmetric so we only construct $\beta$. For grading reasons, $\beta=M_\beta y^\vee+ B$ for some $B\in \Span_{\Mat_n (\F)}\{a_i^\vee\vert \;|a_i|=-1\}$. Reviewing Equation \ref{m22} and abbreviating $\alpha$ as $M y^\vee +A$, we see $-y^\vee=m_2(\alpha,\beta)$ is equivalent to
\[-y^\vee=-M_\beta M y^\vee -B\cdot M-M_\beta A+m_2(A,B).\]
We need to define $M_\beta$ and $B$. It is clear that we must define $M_\beta=M^{-1}$. With this choice of $M_\beta$ we see that any choice for $B$ must satisfy $B=-M^{-1}A\cdot M^{-1}+m_2(A,B)\cdot M^{-1}$. We construct such a $B$ via induction on hieght as follows. If we insist on labeling our Reeb chords with respect to the height filtration, the characterization of $m_2$ implies that $\langle m_2(A,B),a_i^\vee\rangle$ is completely determined by $A$ and those $B_ja_j^\vee$ with $j<i$, and so $B_i$ can be defined by induction. This completes the construction of $\beta$  
To verify that $m_1(\beta)=0$ we appeal to the $A_\infty$ relations. We have
\[m_1(-y^\vee)=m_1(m_2(\alpha,\beta))=m_2(m_1(\alpha),\beta)+m_2(\alpha,m_1(\beta))\]
Observe that $m_1(-y^\vee)$ vanishes since we are computing $m_1(-y^\vee)$ in the self hom space $\Hom(\rho_1,\rho_1)$ (see Lemma \ref{cocycle}). We also know that $m_2(m_1(\alpha),\beta)=0$ since $\alpha$ is a cocycle. So the above equation reduces to $0=m_2(\alpha,m_1(\beta))$. Using $\F$-linearity, we have
\begin{align*}
0=m_2(My^\vee+A,m_1(\beta))&=m_2(My^\vee,m_1(\beta))+m_2(A,m_1(\beta))\\
&=-m_1(\beta)\cdot M+m_2(A,m_1(\beta))
\end{align*}
Rearranging we obtain $m_1(\beta)=m_2(A,m_1(\beta))\cdot M^{-1}$ which implies that $m_1(\beta)\in \Span_{\Mat_n(\F)}\{a_i\}$ by the characterization of $m_2$. 
Using the first bullet point of Lemma \ref{charm1}, one verifies by induction that $\langle m_1(\beta),a_i^\vee\rangle=0$, and so $m_1(\beta)=0$.
\end{proof}
\section{Counting in the Representation Category}\label{counting}
With the equivalence of representations in hand we can start toward the homotopy cardinality of the representation category. Define the homotopy cardinality as one would expect:
\[\#\pi_{\geq0}\rep_n(\Lambda,\F_q)^*=\sum_{[\rho]\in\rep_n(\Lambda,\F_q)/\sim}\frac{1}{|\Aut(\rho)|}\cdot\frac{|H^{-1}\Hom(\rho,\rho)||H^{-3}\Hom(\rho,\rho)|\cdots}{|H^{-2}\Hom(\rho,\rho)||H^{-4}\Hom(\rho,\rho)|\cdots}\] where the equivalence is equivalence in the cohomology category, which we saw from the previous section is the same as conjugate DGA homotopy.

\begin{lemma}\label{uniquerep}
    For $\rho:\alg\to \Mat_n(\F)$ and $M\in \GL_n(\F)$ both fixed, and $A_i\in \Mat_n(\F)$ arbitrarily chosen, then there exists a unique representation $\rho_0:\alg\to \Mat_n(\F)$ and $(\rho,M\rho_0M^{-1})$-derivation $L$ such that $L(a_i)=A_i$ and $M\rho_0M^{-1}\sim\rho$.
\end{lemma}
\begin{proof}
    This is another induction using the height filtration. Suppose $\rho_0$ and $L$ have been defined to satisfy the desired properties on $\mathcal{F}_{i-1}\alg$, the subalgebra of $\alg$ generated by $a_1,\ldots a_{i-1}$. Define $\rho_0(a_i):=M^{-1}\rho(a_i)M+M^{-1}LM(\dd(a_i))$. By construction $M\rho_0(a_i)M^{-1}-\rho(a_i)=L(\dd(a_i)) $. Thus to verify that $\rho_0$ is a representation it only remains to show that $\rho_0:(\mathcal{F}_i\alg,\dd)\to (\Mat_n(\F),0)$ is a chain map and since $\Mat_n(\F)$ has the zero differential this amounts to showing $\rho_0(\dd(a_i))=0$. Since $\dd(a_i)\in\mathcal{F}_{i-1}\alg$ the inductive hypothesis implies that
    \[\rho_0(\dd(a_i))=M^{-1}\rho(\dd(a_i))M+M^{-1}LM(\dd^2(a_i))=0.\]
    Here the first term in the above equation vanishes since $\rho\circ\dd=0$ (the second term is 0 since $\dd$ is a differential).
\end{proof}
\begin{corollary}\label{countunits}
    For any representation $\rho\in\rep_n(\Lambda,\F_q)$ the cardinality of 
    \[ \bigsqcup_{\rho_0\in\rep_n(\Lambda,\F_q)} \{\alpha \in \Hom^0(\rho,\rho_0)\vert m_1(\alpha)=0 \text{ and } [\alpha]  \text{ is an isomorphism in } H^0\Hom(\rho,\rho_0) \}\]
    is $q^{n^2r}|\GL_n(\F_q)|$ where $r$ denotes the number of Reeb chords in DGA degree $-1$.
\end{corollary}
\begin{proof}
    A cocyle in $H^0\Hom_+(\rho,\rho_0)$ representing an isomorphism has the form $\alpha=My^\vee+A$ for $A\in \Span_{\Mat_n(\F_q)}\{a_i^\vee | |a_i|=-1\}$ and $M\in \GL_n(\F_q)$. We see that there are $|\GL_n\F_q|$ choices for $M$, and after fixing $M$, we obtain $q^{n^2r}$ isomorphisms (corresponding to the possible values of the homotopy operator $L$ on those elements of $\alg$ in grading $-1$) from $\rho$ to some $\rho_0$ (with $\rho_0$ uniquely determined by $\alpha$). 
\end{proof}

Lets denote by $\Hom^0(\rho,\rho_0)^\times$ the cocycles in $\Hom^0(\rho,\rho_0)$ that have left and right inverses under the $m_2$ operation. A result found in \cite{NRSS} immediately generalizes to show that there is a bijection from $\Hom^0(\rho,\rho_0)^\times$ to $\Hom^0(\rho,\rho)^\times$. We record their arguments here and explain the modification to our setting. Assume, that the the labeling of the Reeb chords of $\Lambda$ respects the height filtration and define a descending filtration by
\begin{align*}
    \mathcal{F}^{-1}\Hom(\rho_1,\rho_2)&=\Hom(\rho_1,\rho_2)\\
    \mathcal{F}^{0}\Hom(\rho_1,\rho_2)&=\Span_{\Mat_n(\F_q)}\{x^\vee, a_1^\vee, a_2^\vee,\ldots, a_n^\vee\}\\
    \mathcal{F}^{i}\Hom(\rho_1,\rho_2)&=\Span_{Mat_n(\F_q)}\{a_i^\vee, a_{i+1}^\vee,\ldots, a_n^\vee\}.
\end{align*}
Then the $A_\infty$ operations respect this descending filtration on $\rep_n(\Lambda,\F)$ in the following way.

\begin{proposition}\label{filtlemma}\cite{NRSS}
    For $k\geq1$, and any representations $\rho_1,\rho_2,\ldots\rho_{k+1}$ and any integers $1\leq i_1,\ldots, i_k<n$, we have
    \[m_k(\mathcal{F}^{i_k}\Hom(\rho_k,\rho_{k+1})\otimes\cdots \otimes \mathcal{F}^{i_1}\Hom(\rho_1,\rho_2))\subset \mathcal{F}^I\Hom(\rho_1,\rho_{k+1})\]
    where $I=\begin{cases}
        \max\{i_1,\ldots,i_k\}+1 & k\neq2\\
        \max\{i_1, i_2\} &k=2
    \end{cases}$
\end{proposition}
\begin{proof}
    This follows exactly as in \cite[Proposition 9]{NRSS}. We only sketch their proof here for completeness. The idea there is that one can label the generators for the augmented DGA of the $k+1$ copy  $ \alg(\Lambda^{(k+1)})^{\bm{\rho}}$ with respect to increasing height $h(y)<h(x)<h(a_1)<\ldots< h(a_n)$ (here we are purposefully omitting superscripts on our generators) to get the usual filtration for the $k+1$ copy.
    This then translates to the statement that $m_k(\mathcal{F}^{i_k}\Hom(\rho_k,\rho_{k+1})\otimes\cdots \otimes \Hom(\rho_1,\rho_2))\subset \mathcal{F}^I\Hom(\rho_1,\rho_{k+1})$ with $I=\max\{i_1,\ldots,i_k\}$. From 
    here one improves this for $k\neq 2$, by examining the (augmented) DGA of the $k+1$-copy and noting that when $s^{1,k+1}\in \mathcal{F}_p\alg(\Lambda_f^{k+1})^{\bm{\rho}} \setminus \mathcal{F}_{p-1}\alg(\Lambda_f^{k+1})^{\bm{\rho}}$ the only generators in $\mathcal{F}_p\alg(\Lambda_f^{k+1})^{\bm{\rho}} \setminus \mathcal{F}_{p-1}\alg(\Lambda_f^{k+1})^{\bm{\rho}}$ appearing in $\dd^{\bm{\rho}}(s)$ appear in quadratic terms by Proposition 2.4.
\end{proof}
\begin{proposition}\label{bij}
    Suppose $\rho$ and $\rho_0$ are isomorphic in $\rep_n(\Lambda,\F)$ then there is a bijection from $\Hom^0(\rho,\rho_0)^\times$ to $\Hom^0(\rho,\rho)^\times$.
\end{proposition}
\begin{proof}
    We repeat the argument in \cite[Proposition 10]{NRSS} with a mild alteration to account for the $m_k$ operations failing to be $\Mat_n(\F_q)$-multilinear. Fix $f\in \Hom^0(\rho,\rho_0)^\times$ and a left inverse $g\in \Hom^0(\rho_0,\rho)^\times$ for $f$. The define $M_g: \Hom^0(\rho,\rho_0)\to \Hom^0(\rho,\rho)$ and $M_f: \Hom^0(\rho,\rho)\to \Hom^0(\rho,\rho_0)$ by $M_g(\alpha)=m_2(g,\alpha)$ and $M_f(\alpha)=m_2(f,\alpha)$. One can compute using the $A_\infty$ relations that \[M_g\circ M_f(\alpha)=\alpha+m_1(m_3(g,f,\alpha))+m_3(g,f,m_1(\alpha))\]

    We now apply Proposition \ref{filtlemma}, to see that the matrix of $M_g\circ M_f$ with respect to the basis 
    \begin{align*}
        \{E_{ij}y^\vee,E_{ij}a_1^\vee,\ldots,E_{ij}a_m^\vee\}\setminus\{E_{ij}a_k^\vee: |a_k^\vee|\neq0\}
    \end{align*}
    is of the form $I+N$ with $N$ strictly lower triangular. Here, $E_{ij}$ denotes the $n\times n$ matrix with $(i,j)$-entry equal to $1$ and zeros elsewhere, and the basis is ordered lexographically by Reeb chord, followed by $i$, then $j$. This implies that $M_g\circ M_f$ is injective and so $M_f$ is injective. A similar argument shows $M_g$ is injective. 

    It suffices to show $M_f(\Hom^0(\rho,\rho_0)^\times)\subset \Hom^0(\rho,\rho)^\times$. Let $\beta\in \Hom^0(\rho,\rho_0)^\times$. Then
    $m_1(M_f(\beta))=m_1(m_2(f,\beta))=m_2(m_1(f),\beta)+(-1)^{|f|}m_2(f,m_1(\beta))=0$
    where at the last equality we are using the fact that $f$ and $\beta$ are cocycles.
    To see that $M_f(\beta)$ has left and right inverses we recall that having left and right inverses is equivalent to requiring  the coefficient of $y^\vee$ to be invertible. Since $f$ and $\beta$ by definition have invertible $y^\vee$ coefficient we see by Equation \ref{m22} in Proposition \ref{equiv} that so does $M_f(\beta)$. Therefore, $M_f:\Hom^0(\rho,\rho_0)^\times\to \Hom^0(\rho,\rho)^\times$ is an injection and similarly so is $M_g: \Hom^0(\rho,\rho_0)^\times\to \Hom^0(\rho,\rho)^\times$. It follows that that there is a bijection from $\Hom^0(\rho,\rho_0)^\times$ to $\Hom^0(\rho,\rho)^\times$.
\end{proof}

\begin{corollary}\label{cor:countrep}
    Fix $\rho:\alg\to \Mat_n\F_q$ and let $B^0(\rho,\rho)$ denote the coboundaries in $\Hom^0(\rho,\rho)$ and $|\Aut(\rho)|$ denote the number of units in $H^*\Hom(\rho,\rho)$. Then
    \[\vert\{\rho_0:\alg\to \Mat_n(\F_q)| \rho\cong \rho_0\}\vert=\frac{|\GL_n(\F_q)|}{|\Aut(\rho)|}q^{\dim_{\F_q}\Hom^0(\rho,\rho)-\dim_{\F_q}B^0(\rho,\rho)-n^2}\]
\end{corollary}
\begin{proof}
    Corollary \ref{countunits} followed by Proposition \ref{bij} implies that
    \[q^{n^2r}|\GL_n(\F_q)|=\left\vert\bigsqcup_{\rho_0\cong\rho}\Hom^0(\rho,\rho_0)^\times\right\vert=\vert\{\rho_0:\rho\cong \rho_0\}\vert\cdot\vert\Hom^0(\rho,\rho)^\times\vert\] were $r$ is the number of Reeb chords in DGA degree $-1$. Now since $r=\textrm{rank}_{\Mat_n(\F_q)}\Hom^0(\rho,\rho)-1$ and \[|\Hom^0(\rho,\rho)^\times|=|\Aut(\rho)|\cdot |B^0(\rho,\rho)|=|\Aut(\rho)|\cdot q^{\dim_{\F_q}B^0(\rho,\rho)}\] 
    we get 
    \begin{align*}\vert\{\rho_0:\alg\to \Mat_n(\F_q)| \rho\cong \rho_0\}\vert&=\frac{|\GL_n(\F_q)|}{|\Aut(\rho)|}q^{n^2( rk\Hom^0(\rho,\rho)-1)}q^{-\dim_{\F_q}B^0(\rho,\rho)}\\
    &=\frac{|\GL_n(\F_q)|}{|\Aut(\rho)|}q^{\dim_{\F_q}\Hom^0(\rho,\rho)-\dim_{\F_q}B^0(\rho,\rho)-n^2}
    \end{align*}
\end{proof}
Let $r_i$ denote the number of Reeb chords of $\Lambda$ in grading $i$. We define the shifted Euler characteristic $\chi_*$ by \begin{equation}\label{eq: chi}
    \chi_*:=\sum_{i\geq0}(-1)^ir_i+\sum_{i<0}(-1)^{i+1}r_i\end{equation}
\begin{theorem}\label{thm:counting}
    For a Legendrian knot $\Lambda$ in $\R^3$ we have
    \[\#\pi_{\geq0}\rep_n(\Lambda,\F_q)^*=q^{n^2(\frac{tb-\chi_*}{2})}|\GL_n(\F_q)|^{-1}\cdot|\{\rho:\alg\to \Mat_n(\F_q)\}|\] where \[\chi_*:=\sum_{i\geq0}(-1)^ir_i+\sum_{i<0}(-1)^{i+1}r_i\]
\end{theorem}

\begin{proof}
    Let $m_1^i:\Hom^i(\rho,\rho)\to \Hom^{i+1}(\rho,\rho)$, and observe that \begin{align*}
        \dim_{\F_q}\Hom^i(\rho,\rho)&=\dim_{\F_q}\ker m_1^i+\dim_{\F_q}\im m_1^i\\
        &=\dim_{\F_q}\ker m_1^i+\dim_{\F_q}\im m_1^i-\dim_{\F_q}\im m_1^{i-1}+\dim_{\F_q}\im m_1^{i-1}\\
        &=\dim_{\F_q} H^i\Hom(\rho,\rho)+\dim_{\F_q}B^{i+1}(\rho,\rho)+\dim_{\F_q}B^i(\rho,\rho)
    \end{align*}
    Therefore,
    \begin{equation}\label{thmeq1}
        \sum_{i<0}(-1)^i\dim_{\F_q}\Hom^i(\rho,\rho)=-\dim_{\F_q}B^0(\rho,\rho)+\sum_{i<0}(-1)^i\dim_{\F_q}H^i\Hom(\rho,\rho).
        \end{equation}
        Meanwhile, we have \begin{equation}\label{thmchi}
            n^2\chi_*=-2n^2+\sum_{i\leq0}(-1)^i\dim_{\F_q}\Hom^i(\rho,\rho)-\sum_{i>0}(-1)^i\dim_{\F_q}\Hom^i(\rho,\rho)\end{equation} and \begin{equation}\label{thmtb}
            -n^2tb=\sum_i(-1)^i\dim_{\F_q}\Hom^i(\rho,\rho)
            \end{equation} where the $-2n^2$ in Equation \ref{thmchi} accounts for extra contributions coming from $y^\vee$ and $x^\vee$. Combining the above equations we see that \begin{align*}
                n^2\left(\frac{\chi_*-tb}{2}\right)\overset{\ref{thmchi}\&\ref{thmtb}}&{=}-n^2+\dim_{\F_q}\Hom^0(\rho,\rho)+\sum_{i<0}(-1)^i\dim_{\F_q}\Hom^i(\rho,\rho)\\
                &\overset{\ref{thmeq1}}{=}-n^2+\dim_{\F_q}\Hom^0(\rho,\rho)-\dim_{\F_q}B^0(\rho,\rho)+\sum_{i<0}(-1)^i\dim_{\F_q}H^i\Hom^i(\rho,\rho)
                \end{align*}
           Therefore the statement of Corollary \ref{cor:countrep} can be rephrased as
           \begin{align*}
               |\{\rho:\rho\cong\rho_0\}|&=\frac{|\GL_n(\F_q)|}{|\Aut(\rho)|}q^{n^2\left(\frac{\chi_*-tb}{2}\right)}q^{-\sum_{i<0}(-1)^i\dim_{\F_q}H^i\Hom^i(\rho,\rho)}\\
               &=\frac{|\GL_n(\F_q)|}{|\Aut(\rho)|}q^{n^2\left(\frac{\chi_*-tb}{2}\right)} \cdot \frac{|H^{-1}\Hom(\rho,\rho)||H^{-3}\Hom(\rho,\rho)|\cdots}{|H^{-2}\Hom(\rho,\rho)||H^{-4}\Hom(\rho,\rho)|\cdots}.
               \end{align*}
            Rearranging and summing over isomorphism class of representations gives us 
            \begin{align*}
                \#\pi_{\geq0}\rep_n(\Lambda,\F_q)^*&=\sum_{[\rho]\in\rep_n(\Lambda,\F_q)/\sim}\frac{1}{|\Aut(\rho)|}\cdot\frac{|H^{-1}\Hom(\rho,\rho)||H^{-3}\Hom(\rho,\rho)|\cdots}{|H^{-2}\Hom(\rho,\rho)||H^{-4}\Hom(\rho,\rho)|\cdots}\\
                &=q^{n^2(\frac{tb-\chi_*}{2})}|\GL_n(\F_q)|^{-1}\cdot|\{\rho:\alg\to \Mat_n(\F_q)\}|
            \end{align*} as desired.
\end{proof}

We now recall the definition of an $m$-graded normal ruling. A normal ruling of $\Lambda$ is a decomposition of the front diagram into a collection of simple closed curves such that each simple closed curve has exactly one corner at a right cusp and exactly one corner at a left and possibly has corners at crossings (called \emph{switches}) subject to the normality condition in Figure \ref{fig:normality}.

\begin{figure}[ht]
    \centering
    \includegraphics{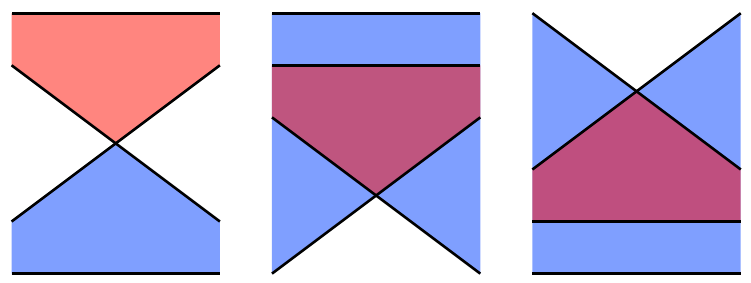}
    \caption{The normality condition says that at a switch the interiors of the corresponding simple closed curves are locally disjoint or nested.}
    \label{fig:normality}
\end{figure}

Further suppose that $\Lambda$ comes equipped with a $\Z/m$-valued Maslov potential\footnote{A Maslov potential is a locally constant function on the front diagram of $\Lambda$ that increments by one as one traverses up through a cusp} for $m|2r(\Lambda)$. If all switches of a ruling $\tau$ occur between strands of equal Maslov potential (mod $m$) then $\tau$ is called $m$-\emph{graded}.
\begin{definition}\label{ruling}
The $m$-graded ruling polynomial of $\Lambda$ is defined by \[R^m_{\Lambda}(z)=\sum_{\tau}z^{|\mathrm{switches}(\tau)|-|\mathrm{right \;cusps}(\tau)|}\]
\end{definition}
\begin{definition}\label{coloredruling}\cite{LR}
    For $m\neq 1$, the $m$-graded $n$-colored ruling polynomial is defined by \[R^m_{n,\Lambda}(q)=\frac{1}{c_n}\sum_{\beta\in S_n}q^{\ell(\beta)/2}R^m_{S(\Lambda,\beta)}(z)\vert_{z=q^{1/2}-q^{-1/2}}\]
    where 
        $S(\Lambda,\beta)$ denotes the satellite construction of $\Lambda$ by a positive permutation braid\footnote{A positive permutation braid is the unique Legendrian $n$-tangle characterized  by requiring the front projection of $\beta$ has no cusps, and (labeling strands top to bottom) if $i<j$ then then $\beta$ has no (resp. exactly one) crossings between strands labeled $i$ and $j$ at $x=0$ if $\beta(i)<\beta(j)$ (respectively $\beta(i)>\beta(j)$).}  $\beta\subset J^1[0,1]$ with Maslov potential 0 on all strands of $\beta$;
        $\ell(\beta)$ is the length of $\beta$; and 
        \[c_n=(q^{1/2})^{n(n-1)/2}\prod_{i=1}^n\frac{q^{i/2}-q^{-i/2}}{q^{1/2}-q^{-1/2}}.\]
    
\end{definition}
In \cite{LR}, Legendrian invariant $m$-graded total $n$-dimensional representation numbers were defined for any $m\geq0$ in full generality. We record their definition (in our restricted context).

\begin{definition}
    Let $\Lambda\subset J^1\R$ have $r(\Lambda)=0$. Then the total $0$-graded $n$-dimensional representation number is defined by \[Rep_0(\Lambda,\F_q^n)=q^{-n^2(\frac{\chi_*}{2})}|\GL_n(\F_q)|^{-1}\cdot |\{\rho:\alg\to \Mat_n(\F_q)\}|\]
\end{definition}

Combining Theorem \ref{thm:counting} with 
\begin{theorem}\cite{LR}
    Let $\Lambda\subset J^1(\R)$ be a Legendrian with $r(\Lambda)=0$, then
    \[Rep_0(\Lambda,\F_q^n)=R^0_{n,\Lambda}(q)\]
\end{theorem}
one gets
\begin{corollary}For a Legendrian knot $\Lambda\subset \R^3$ with $r(\Lambda)=0$, the $n$-colored ruling polynomial of $\Lambda$ is defined by
    \[\#\pi_{\geq0}\rep_n(\Lambda,\F_q)^*=q^{n^2tb/2}R^0_{n,\Lambda}(q).\] 
\end{corollary}


\section{Applications to Concordance}\label{conc}

In this section we explore immediate results arising from the previous sections and the current literature.
We do this by combining results found in \cite{CorNS}, \cite{CLL+}, and \cite{Pan}. We recall some of their results.

\begin{theorem}\cite[Theorem 2.4]{CorNS}
    Suppose $L$ is an exact Lagrangian concordance from $\Lambda_-$ to $\Lambda _+$, then for any pattern $\beta\subset J^1(S^1)$ there exists an exact Lagrangian concordance from $S(\Lambda_-,\beta)$ to $S(\Lambda_+,\beta)$
\end{theorem}
\begin{remark}
    In \cite{CorNS} they consider concordance between knots, and so they restrict to the case where $\beta$ is connected. However, their argument, which is an application of the Weinstein neighborhood theorem, pushes through to the case where $\beta$ is disconnected.
\end{remark}
\begin{theorem}\cite{CLL+}
Let $L$ be a spin exact Maslov-0 Lagrangian concordance from $\Lambda_-$ to $\Lambda_+$ then $\#\pi_{\geq0}\aug_+\Lambda_-,\F_q)^*\leq \#\pi_{\geq0}\aug_+(\Lambda_+,\F_q)^*$ and $R^0_{\Lambda_-}(q^{1/2}-q^{-1/2})\leq R^0_{\Lambda_+}(q^{1/2}-q^{-1/2})$ 
\end{theorem}

\begin{corollary} \label{appl}   
If $L$ is a spin exact Lagrangian concordance form $\Lambda_-$ to $\Lambda_+$ and either $r(\Lambda_-)=0$ or $r(\Lambda_+)=0$, then $\#\pi_{\geq0}\rep_n(\Lambda_-,\F_q)^*\leq \#\pi_{\geq0}\rep_n(\Lambda_+,\F_q)^*$
and $R^0_{n,\Lambda_-}(q)\leq R^0_{n,\Lambda_+}(q)$
\end{corollary}
\begin{proof}
It suffices to show that the concordance from $S(\Lambda_-,\beta)=S(\Lambda_+,\beta)$ is Maslov-0 for each positive permutation braid $\beta$. This is follows from the observation that if $\Lambda$ is a Legendrian on either end of the concordance $C$ then we have the following commutative diagram relating $\Lambda$, $C$, and the Lagrangian Grassmannians $LGr(-)$

\[ \begin{tikzcd}
H_1(C) \arrow{r}{i_*} 
\arrow["\sim", "r_*" ']{d}
& H_1(LGr(2)) \arrow{rr}{\mu}
    &&\Z \arrow[equal]{d}\\%
H_1(\Lambda) \arrow{r}{j_*}& H_1(LGr(1))\arrow[swap]{u}{(id_{LGr(1)}\times \dd_{x_2})_*}\arrow{rr}{\mu} &&\Z
\end{tikzcd}
\]
where $i$ and $j$ are inclusions and $r$ is a deformation retract. Since the bottom composition is 0 ($r(\Lambda)=0$) we see that the Maslov number of $C$ is also 0. The corollary then follows from applying the above theorems to each of $R_{S(\Lambda_\pm,\beta)}$ summands.
\end{proof}
Corollary \ref{appl}, also has some interesting consequences. For example one can see that if there are Maslov-0 concordances from $U$ to $\Lambda$ and $\Lambda$ to $U$, then the homotopy cardinality of $\Lambda$ is exactly the same as $U$. So, if $U$ is the max $tb$ unknot, then 
\[\#\pi_{\geq0}\rep_n(\Lambda,\F_q)^* =\#\pi_{\geq0}\rep_n(U,\F_q)^*=|GL_n(\F_q)|^{-1}.\]

\subsection{Further directions}
It is unclear, whether the homotopy cardinality of the representation category is a strictly stronger obstruction to concordance than that of the augmentation category. More precisely, is there a Maslov-0 concordance from $\Lambda_-$ to $\Lambda_+$ where both $\Lambda_\pm$ have no augmentations and such that the inequalities in Corollary \ref{appl} are strict for some $n>1$? An obvious strategy to tackle this would be as follows. First, find a knot $\Lambda_n$ that has no (0-graded) augmentations but has higher $n$-dimensional (0-graded) representations. Pick any two Maslov-0 concordant Legendrians $\wt{\Lambda_-}$ and $\wt{\Lambda_+}$ with 
\[\#\pi_{\geq0}\aug(\wt{\Lambda_-},\F_q)^* <\#\pi_{\geq0}\aug(\wt{\Lambda_+},\F_q)^*\]
and then one takes $\Lambda_\pm$ to be $\Lambda_n\#\wt{\Lambda_\pm}$. Some auxiliary choices need to be made in order for this strategy to be viable, for example, the choice the Maslov potentials on $\Lambda_n$ and $\wt{\Lambda_-}$ (respectively $\wt{\Lambda_+}$) should be chosen to be compatible with a Maslov potential on $\Lambda_n\#\wt{\Lambda_-}$ (respectively $\Lambda_n\# \wt{\Lambda_+}$). It is known that there are Legendrians with no augmentations but have 2-dimensional ungraded representations namely Sivek's negative torus knots. The family $\Lambda_n$ is the exact same family proposed in \cite{CDGG} which can distinguish mirrors within connected sums.

Our setting is somewhat restrictive. We have chosen to work with connected Legendrians with a single basepoint, and a single dip and concordances between such Legendrians. There are several ways this could be generalized. For example, one could ask whether similar results hold for links. In addition, Lemma \ref{cocycle} holds for $2m$-graded representations, so another direction could be to address how much of the above story works for say $2m$-graded representations. In fact, there is still some outstanding work in the augmentation setting, as it is rather unclear how to define the homotopy cardinality for the $2m$ graded augmentation category. For example, In \cite[Conjecture 19]{NRSS} they conjecture \[2\dim\Hom^{0}(\epsilon,\epsilon)-2\dim B^{0}(\epsilon,\epsilon)-1
-\chi_*(\Lambda) =
2\dim H^0\Hom(\epsilon,\epsilon)-\dim H^1\Hom(\epsilon,\epsilon).\] Assuming this, the authors deduce what the appropriate definition for the $2m$-graded homotopy cardinality should be, that is what is correct combination of cohomology groups that gives a multiple of the ruling polynomial. Under the assumption of the previous conjecture one obtains \cite[Corollary 20]{NRSS} which states that for any $m\geq 0$, the homotopy cardinality of the $2m$-graded augmentation category should be
\[
\#\pi_{\geq0}\aug_+(\Lambda,\F_q)^*:=\sum_{[\epsilon]} \frac{1}{\vert\Aut(\epsilon)\vert} \frac{|H^0\Hom(\epsilon,\epsilon)|}{|H^1\Hom(\epsilon,\epsilon)|^{1/2}} q^{(tb(\Lambda)-1)/2} = q^{tb/2} R^{2m}_\Lambda(q^{1/2} - q^{-1/2}).
\] So one might suspect under a verification of a similar conjecture to obtain a definition that has $q^{n^2tb/2}R^{2m}_{n,\Lambda}(q)$ on the right hand side.

\bibliography{output}  

\newcommand{\etalchar}[1]{$^{#1}$}
\begin{thebibliography}{CDRGG16}

\bibitem[BC14]{BC}
Fr\'{e}d\'{e}ric Bourgeois and Baptiste Chantraine.
\newblock Bilinearized {L}egendrian contact homology and the augmentation
  category.
\newblock {\em J. Symplectic Geom.}, 12(3):553--583, 2014.

\bibitem[BD01]{BD01}
John~C. Baez and James Dolan.
\newblock From finite sets to {F}eynman diagrams.
\newblock In {\em Mathematics unlimited---2001 and beyond}, pages 29--50.
  Springer, Berlin, 2001.

\bibitem[CDRGG16]{CDGG}
Baptiste Chantraine, Georgios Dimitroglou~Rizell, Paolo Ghiggini, and Roman
  Golovko.
\newblock Noncommutative augmentation categories.
\newblock In {\em Proceedings of the {G}\"{o}kova {G}eometry-{T}opology
  {C}onference 2015}, pages 116--150. G\"{o}kova Geometry/Topology Conference
  (GGT), G\"{o}kova, 2016.

\bibitem[CNS16]{CorNS}
Christopher Cornwell, Lenhard Ng, and Steven Sivek.
\newblock Obstructions to {L}agrangian concordance.
\newblock {\em Algebr. Geom. Topol.}, 16(2):797--824, 2016.

\bibitem[CNS19]{CNS18}
Baptiste Chantraine, Lenhard Ng, and Steven Sivek.
\newblock Representations, sheaves and {L}egendrian {$(2,m)$} torus links.
\newblock {\em J. Lond. Math. Soc. (2)}, 100(1):41--82, 2019.

\bibitem[CSLL{\etalchar{+}}22]{CLL+}
Orsola Capovilla-Searle, Noémie Legout, Maÿlis Limouzineau, Emmy Murphy,
  Yu~Pan, and Lisa Traynor.
\newblock Obstructions to reversing lagrangian surgery in lagrangian fillings,
  2022.

\bibitem[EN22]{EN}
John~B. Etnyre and Lenhard~L. Ng.
\newblock Legendrian contact homology in {$\mathbb{ R}^3$}.
\newblock In {\em Surveys in differential geometry 2020. {S}urveys in
  3-manifold topology and geometry}, volume~25 of {\em Surv. Differ. Geom.},
  pages 103--161. Int. Press, Boston, MA, [2022] \copyright 2022.

\bibitem[FHT15]{FHT}
Yves F\'{e}lix, Steve Halperin, and Jean-Claude Thomas.
\newblock {\em Rational homotopy theory. {II}}.
\newblock World Scientific Publishing Co. Pte. Ltd., Hackensack, NJ, 2015.

\bibitem[FI04]{FI}
Dmitry Fuchs and Tigran Ishkhanov.
\newblock Invariants of {L}egendrian knots and decompositions of front
  diagrams.
\newblock {\em Mosc. Math. J.}, 4(3):707--717, 783, 2004.

\bibitem[Fuc03]{Fuchs}
Dmitry Fuchs.
\newblock Chekanov-{E}liashberg invariant of {L}egendrian knots: existence of
  augmentations.
\newblock {\em J. Geom. Phys.}, 47(1):43--65, 2003.

\bibitem[HR15]{HR}
Michael~B. Henry and Dan Rutherford.
\newblock Ruling polynomials and augmentations over finite fields.
\newblock {\em J. Topol.}, 8(1):1--37, 2015.

\bibitem[LR20]{LR}
Caitlin Leverson and Dan Rutherford.
\newblock Satellite ruling polynomials, {DGA} representations, and the colored
  {HOMFLY}-{PT} polynomial.
\newblock {\em Quantum Topol.}, 11(1):55--118, 2020.

\bibitem[MP07]{MP07}
Jo\~{a}o~Faria Martins and Timothy Porter.
\newblock On {Y}etter's invariant and an extension of the {D}ijkgraaf-{W}itten
  invariant to categorical groups.
\newblock {\em Theory Appl. Categ.}, 18:No. 4, 118--150, 2007.

\bibitem[MR20]{MR}
Justin Murray and Dan Rutherford.
\newblock Legendrian {DGA} representations and the colored {K}auffman
  polynomial.
\newblock {\em SIGMA Symmetry Integrability Geom. Methods Appl.}, 16:Paper No.
  017, 33, 2020.

\bibitem[NRS{\etalchar{+}}20]{NRSSZ}
Lenhard Ng, Dan Rutherford, Vivek Shende, Steven Sivek, and Eric Zaslow.
\newblock Augmentations are sheaves.
\newblock {\em Geom. Topol.}, 24(5):2149--2286, 2020.

\bibitem[NRSS17]{NRSS}
Lenhard Ng, Dan Rutherford, Vivek Shende, and Steven Sivek.
\newblock The cardinality of the augmentation category of a {L}egendrian link.
\newblock {\em Math. Res. Lett.}, 24(6):1845--1874, 2017.

\bibitem[NS06]{NS}
Lenhard~L. Ng and Joshua~M. Sabloff.
\newblock The correspondence between augmentations and rulings for {L}egendrian
  knots.
\newblock {\em Pacific J. Math.}, 224(1):141--150, 2006.

\bibitem[Pan17]{Pan}
Yu~Pan.
\newblock The augmentation category map induced by exact {L}agrangian
  cobordisms.
\newblock {\em Algebr. Geom. Topol.}, 17(3):1813--1870, 2017.

\bibitem[PC05]{PC}
P.~E. Pushkar' and Yu.~V. Chekanov.
\newblock Combinatorics of fronts of {L}egendrian links, and {A}rnol'd's
  4-conjectures.
\newblock {\em Uspekhi Mat. Nauk}, 60(1(361)):99--154, 2005.

\bibitem[Sab05]{Sab}
Joshua~M. Sabloff.
\newblock Augmentations and rulings of {L}egendrian knots.
\newblock {\em Int. Math. Res. Not.}, (19):1157--1180, 2005.

\end{thebibliography}
\bibliographystyle{alpha}

\end{document}